\documentclass[12pt]{amsart}
\newtheorem{thm}[subsection]{Theorem}
\newtheorem{defn}[subsection]{Definition}

\newtheorem{prop}[subsection]{Proposition}
\theoremstyle{definition}

\numberwithin{equation}{section}

\begin{document}

\title[Weyl-von Neumann-Berg theorem for  quaternionic operators ]{Weyl-von Neumann-Berg theorem for  quaternionic operators  }

\author{G. Ramesh}

\address{Department of Mathematics\\I. I. T. Hyderabad,  Kandi (V)\\ Sangareddy (M), Medak (D)\\ Telangana , India-502 285.}

\email{rameshg@iith.ac.in}

\thanks{}

\date{\today}


\begin{abstract}
We prove the Weyl-von Neumann-Berg theorem for quaternionic right linear operators (not necessarily bounded) in a quaternionic Hilbert space: Let $N$ be a right linear normal (need not be bounded) operator in a quaternionic separable Hilbert space $H$. Then for a given $\epsilon>0$ there exists a compact operator $K$ with $\|K\|<\epsilon$ and a diagonal operator $D$ on $H$ such that
$N=D+K$.
\end{abstract}
\maketitle
\section{Introduction}
Herman Weyl proved  that every bounded self-adjoint operator on a separable Hilbert space is sum of a diagonal operator and a compact operator. Later, von-Neumann observed that the compact operator can be replaced by a Hilbert-Schmidt operator with arbitrary small norm and the boundedness of the operator can be dropped. Afterwards, Berg extended this result to the case of normal operators defined in a separable complex Hilbert space \cite[Theorem 1]{berg}. Now, this theorem is well known as Weyl-von Neumann-Berg theorem. We refer \cite{davidson,Halmos10probs} for more details on this result.

In this note we prove the Weyl-von Neumann-Berg theorem for right linear normal operators defined in a separable quaternionic Hilbert space with the observation that given  a right linear  operator in a quaternionic Hilbert space can be  associated to  a complex linear  operator which preserves certain properties of the original operator. This can be done if there exists an anti self-adjoint unitary operator which commutes with the given operator. 
There always exists such an operator in case if the operator is normal (see \cite[Theorem 5.9, Proposition 3.11]{Ghiloni} for details for the case of bounded operators). The unbounded case is discussed in \cite{Rameshandsanthosh3} in a general setting. Using the above mentioned technique and some more auxiliary results, we obtain the result for the case of quaternionic operators.

In the remaining part of this section we recall necessary definitions and basic results of quaternionic Hilbert spaces and right linear operators on such spaces. In the second section we prove the main result.
\section{Preliminaries}

We denote the division ring of real quaternions by $\mathbb H$. If $q\in \mathbb H$, then $q = q_{0} + q_{1}i+q_{2}j+q_{3}k$, where $q_{r} \in \mathbb{R}$ for $r = 0, 1, 2, 3$ and $i,j,k$ satisfy the following conditions:
\begin{equation*}			
 i^{2}=j^{2}=k^{2}=-1=ijk.
\end{equation*}
 The conjugate of $q$ is $ \overline{q} = q_{0}-q_{1}i-q_{2}j-q_{3}k$ and $ |q| := \sqrt{q_{0}^{2}+q_{1}^{2}+q_{2}^{2}+q_{3}^{2}}$. The imaginary part of $\mathbb{H}$ is defined by  $\text{Im}(\mathbb{H}) = \left\{ q \in \mathbb{H} : q = -\overline{q}\right\}.$ The set of all unit imaginary quaternions is denoted by $\mathbb{S}$, that is $ \mathbb{S}:= \left\{ q \in \text{Im}(\mathbb{H}): |q| = 1 \right\}$ and  the unit sphere of $H$ by $S_H$.

Here we list out some of the properties of quaternions, which we need later.

  For $p,q \in \mathbb{H}$, we have $ \overline{p.q} = \overline{q}. \overline{p},\; |p.q| = |p|.|q|$ and $|\overline{p}| = |\overline{q}|$.  Define  relation on $\mathbb{H}$ as $p \sim q$ if and only if $ p = s^{-1}qs $ for some $ 0\neq s \in \mathbb{H} $.  Then $''\sim''$ is an equivalence relation and the  equivalence class of $p$ is $  [p]:= \{s^{-1}qs : 0 \neq s \in H\}$. For each $m \in \mathbb{S}, \mathbb{C}_{m}: =\left\{ \alpha + m \beta : \alpha, \beta \in \mathbb{R} \right\} $ is a real subalgebra of $\mathbb{H}$ and is called as the slice complex plane. The upper half slice complex plane is defined by $\mathbb C^{+}_m:={\{\alpha+m\beta: \alpha \geq 0,\; \beta \in \mathbb R}\}$.  For $ m \neq \pm{n},$ we have $\mathbb{C}_{m} \cap \mathbb{C}_{n} = \mathbb{R}$ and $\mathbb{H} = \displaystyle \cup_{m \in \mathbb{S}} \mathbb{C}_{m}$.

A right $\mathbb H$-module $H$ is called a quaternionic pre-Hilbert space if there exists
a Hermitian quaternionic scalar product $\langle\cdot\rangle: H\times H\rightarrow \mathbb H$
satisfying the following properties:
\begin{enumerate}
\item  $\langle u,vp + wq\rangle = \langle u,v\rangle \, p + \langle u,w\rangle \,q$  for all  $u, v,w \in H$ and $p, q \in \mathbb H$
\item  $\langle u,v\rangle = \overline{\langle v,u\rangle}$  for all $u, v \in H$
\item $\langle u,u\rangle\geq 0 $ for all $u\in H$ and $\langle u,u\rangle=0$ iff $u=0$.
\end{enumerate}

 Define $ \| u \|
= {\left\langle u, u \right\rangle}^{\frac{1}{2}}, $ for every $ u \in H.$  Then $\|\cdot\|$ is a norm.  If the normed space $ (H, \|\cdot\|) $ is  complete, then $H$ is called a quaternionic Hilbert space.

We say $H$ is separable if there exists a countable orthonormal basis for $H$. Let ${\{\phi_r:r\in \mathbb N}\}$ be an orthonormal basis for $H$. Then every $x\in H$ has the representation:
\begin{equation*}
x=\displaystyle \sum_{n=1}^{\infty} \phi_n \langle \phi_n,x\rangle.
\end{equation*}

Let $D$ is a right linear subspace of $H$ and $T: D\rightarrow H$ be a map. Then $T$ is said to be right linear if   $T(x+y)=Tx+Ty$  for all $x,y\in D$  and $T(xq)=(Tx)q$ for all $x\in D,\; q\in \mathbb H$. Usually, we denote $D$,
the domain of $T$ by $D(T)$.

The graph of a right linear operator $T$ is denoted by $G(T)$ and is defined as  $G(T) = \left\{(x, Tx) | x \in D(T) \right\}$ . If  the graph $G(T)$ is closed in $H \times H$, then $T$ is said to be a closed operator. Equivalently, if $(x_{n}) \subset D(T)$ with $ x_{n} \to x \in H$ and $Tx_{n} \to y$, then $x \in D(T)$ and $Tx=y$.

	We say  $T$  to be densely defined, if $D(T)$ is dense in $H$. For such an operator, then there exists a unique operator $T^{*} $  with domain $D(T^*)$, where
\begin{equation*}
D(T^*):={\{y\in H: \text{the functional} \; D(T) \ni x \rightarrow \langle y,Tx\rangle \;  \text{ is continuous}}\}
\end{equation*}
and satisfy  $ \left\langle y,  Tx\right\rangle = \left\langle T^{*}y ,x\right\rangle $ for all $ y\in D(T^*)$ and for all $x\in D(T)$.  This operator $ T^{*} $ called the adjoint of $T$.  Clearly, $T^*$ is right linear. Furthermore, $T^*$ is always closed irrespective of $T$.

We denote  the class of densely defined closed right  linear operators in $H$ by $\mathcal{C}(H)$.

	Let $S, T \in \mathcal{C}(H)$  with domains $D(S)$ and $D(T)$, respectively. Then, $S$ is said to be a restriction of $T$ denoted by $S \subseteq T$, if $D(S) \subseteq D(T)$ and $ Sx = Tx $, for all $x \in D(S)$. In this case, $T$ is called an extension of $S$. We say $S=T$ if $S\subseteq T$ and $T\subseteq S$. In other words, $S=T$ if and only if $D(S)=D(T)$ and $Sx=Tx$ for all $x\in D(T)$.

We define the sum as: $(S+T)(x)=Sx+Tx$ for all $x\in D(T)\cap D(S)$.

Let $T \in \mathcal{C}(H) $ and $S \in \mathcal{B}(H)$, then we say that $S$ commute with $T$ if $ST \subseteq TS$. That is, $STx = TSx$, for all $x \in D(T)$.

	 	Let $T \in \mathcal{C}(H)$.  Then $T$ is said to be self-adjoint if $ T = T^{*}$, anti self-adjoint if $T^{*}= -T$, normal if $ TT^{*} = T^{*}T $ and  unitary if $ TT^{*} = T^{*}T = I$.

A right linear operator $T:H\rightarrow H$ is said to be bounded if there exists a $M>0$ such that
$\|Tx\| \leq M\, \|x\|$ for all $x\in H$. For such an operator the norm is defined by
\begin{equation*}
\|T\| = \sup \left\{ \|Tu\| : u \in S_H \right\}.
\end{equation*}
We denote the space of all bounded right linear operators on $H$ by $\mathcal B(H)$.

By the closed graph theorem, if $T\in \mathcal C(H)$ and $D(T)=H$, then $T\in \mathcal B(H)$. Thus, for $T$ unbounded, $D(T)$ is a proper subspace of $H$.

Let $H$ be a separable Hilbert space with an orthonormal basis ${\{\phi_r: r\in \mathbb N}\}$.
 \begin{enumerate}
 \item Let $T\in \mathcal B(H)$. Then $T$ is said to be
 \begin{itemize}
 \item[(a)] \textit{compact} if ${T(B)}$ is pre-compact for every bounded subset $B$ of $H.$ Equivalently, $\{T(x_{n})\}$ has a convergent subsequence for every bounded sequence $\{x_{n}\} $ of $ H$

     \item[(b)] \textit{Hilbert-Schmidt} if $\|T\|_2:=\displaystyle \sum_{r=1}^{\infty} \|T\phi_r\|^2<\infty$. In fact $\|T\|_2$ does not depend on the orthonormal basis of $H$ and is known as the Hilbert-Schmidt norm of $T$.
     \end{itemize}
\item  If $T\in \mathcal C(H)$ is densely defined, then $T$ is said to be \textit{diagonal} with respect to ${\{\phi_r:r\in \mathbb N}\}$ if
there exists a sequence $(q_r)\subset \mathbb H$ such that $T\phi_r=\phi_r \, q_r$ for each $r\in \mathbb N$.
Note that, here $\phi_r\in D(T)$ for each $r\in \mathbb N$. In this case  the matrix of $T$ with respect to $\mathcal B$ is $(a_{rs})$,
where $a_{rs}=\langle \phi_s,T(\phi_r)\rangle =\delta_{rs}q_{s}$ for each $r,s\in \mathbb N$.
\end{enumerate}
We recall the notion of the spherical spectrum of a right linear operator in a quaternionic Hilbert space.
\begin{defn}\cite[Definition 4.1]{Ghiloni}	
Let  $ T \colon D(T) \to H $ and $ q \in \mathbb{H}.$ Define  $\Delta_{q}(T) \colon D(T^{2}) \to H$ by
	\begin{equation*}
	\Delta_{q}(T):= T^{2}-T(q+\overline{q})+I.|q|^{2}.
	\end{equation*}
	The spherical resolvent of $T$ is denoted by $\rho_{S}(T)$ and is the set of all $q \in \mathbb{H}$ satisfying the following three properties:
	\begin{enumerate}
		\item $N(\Delta_{q}(T)) = \{0\}$
		\item $R(\Delta_{q}(T))$ is dense in $H$
		\item $\Delta_{q}(T)^{-1}\colon R(\Delta_{q}(T)) \to D(T^{2})$ is bounded.
	\end{enumerate}
	The spherical spectrum of $T$ is defined by $\sigma_{S}(T) = \mathbb{H} \setminus \rho_{S}(T)$ .
	\end{defn}
\begin{defn}\label{circularization}\cite[Page 46]{Ghiloni}
Let $m\in \mathbb S$ and let $K$ be a compact subset of $\mathbb C_m$. Then we define the circularization $\Omega_K$ of $K$ in $\mathbb H$ by
\begin{equation*}
  \Omega_K:={\{\alpha+j\beta:\alpha, \beta \in \mathbb R, \alpha +m\beta \in K,\; j\in \mathbb S}\}.
  \end{equation*}	
  \end{defn}	
		
Let  $m \in \mathbb{S}$ and $ J \in \mathcal{B}(H)$ be anti self-adjoint, unitary operator. Define $H_{\pm}^{J_m}:={\{u\in H:Ju=\pm um}\}$. Then $H_{\pm}^{J_m}$ is a non-zero closed subset of $H$. The restriction of the inner product on $H$ to $H_{\pm}^{J_m}$ is a $\mathbb C_m$-valued inner product and with respect to this inner product $H_{\pm}^{J_m}$ is a Hilbert space. In fact, if we consider $H$ as a $\mathbb C_m$-linear space, $H$ has the decomposition: $H=H_{+}^{J_m}\oplus H_{-}^{J_m}$ (see \cite[pages 21-22]{Ghiloni} for details).

The following result is crucial in proving our results.
\begin{prop} \cite[Proposition 3.11]{Ghiloni}\label{extensionrestriction}
		If $T \colon D(T) \subset H^{Jm}_{+} \to H^{Jm}_{+} $ is a $\mathbb{C}_{m}-$ linear operator, then there exists a unique right $\mathbb{H}-$ linear operator $
		\widetilde{T}\colon D(\widetilde{T})\subset H \to H $
		such that $ D(\widetilde{T}) \bigcap H^{Jm}_{+} = D(T), \ J(D(\widetilde{T})) \subset D(\widetilde{T})$ and $ \widetilde{T}(u) = T(u), $ for every $u \in H^{Jm}_{+}.$ The following facts holds:
		\begin{enumerate}
			\item If $T \in \mathcal{B}(H^{Jm}_{+})$, then $\widetilde{T}\in \mathcal{B}(H)$ and $\|\widetilde{T}\| = \|T\|$
			\item $J\widetilde{T} = \widetilde{T} J$.
		\end{enumerate}
		On the other hand, let $V \colon D(V) \to H$ be a right linear  operator. Then $ V = \widetilde{U} $, for a unique bounded $\mathbb{C}_{m}-$ linear operator $ U \colon D(V) \bigcap H^{Jm}_{+} \to H^{Jm}_{+} $ if and only if $J(D(V)) \subset D(V)$ and $JV \subseteq  VJ$.
		
		Furthermore,
		\begin{enumerate}
			\item \label{extadjoint}If $\overline{D(T)} = H^{Jm}_{+}$, then $\overline{D(\widetilde{T})} = H$ and  $\big(\widetilde{T}\big)^{*} = \widetilde{T^{*}}$
			\item \label{extnmulti}If $S \colon D(S) \subset H^{Jm}_{+} \to H^{Jm}_{+}$ is $\mathbb{C}_{m}$- linear, then $\widetilde{ST} = \widetilde{S} \widetilde{T}$
			\item \label{extninverse}If $S$ is the inverse of $T$, then $\widetilde{S}$ is the inverse of $\widetilde{T}.$
		\end{enumerate}
	\end{prop}

By (\ref{extadjoint}) and  (\ref{extnmulti}), it follows that if $T_{+}$ is normal, then $T$ is normal. On the other hand, by the uniqueness of the extension it follows that normality of $T$ implies the normality of $T_{+}$.

In particular, if $T\in \mathcal B(H)$ is normal, there exists an anti self-adjoint, unitary $J\in \mathcal B(H)$ such that $TJ=JT$ (see \cite[Theorem 5.9]{Ghiloni} for details). Hence in this case all the statements in Theorem \ref{extensionrestriction} holds true. In case, if $T\in \mathcal C(H)$ is normal, existence of a anti unitary self-adjoint operator $J\in \mathcal B(H)$ such that
$JT\subseteq TJ$ is proved in \cite[Theorem 3.6]{Rameshandsanthosh3}. For the sake of completeness we give the details here.
\begin{thm} \label{commutes}
	Let $T \in \mathcal{C}(H)$ with the domain $D(T) \subseteq H$. Then there exists an anti self-adjoint unitary operator $J$ on $H$ such that $JT\subseteq TJ$.
	\end{thm}
\begin{proof}
Let $\mathcal Z_T:=T(I+T^*T)^{-\frac{1}{2}}$ be the $\mathcal Z$-transform of $T$ (see \cite[Theorem 6.1]{alpayetal1} for details). Since $T$ is normal, so is $\mathcal Z_T$. Since $\mathcal Z_T$ is bounded, by \cite[Theorem 5.9]{Ghiloni}, there exists an anti self-adjoint, unitary $J\in \mathcal B(H)$ such that $J\mathcal Z_T=\mathcal Z_T \,J$. As $T=\mathcal Z_T(I-\mathcal Z^{*}_T\mathcal Z_T)^{-\frac{1}{2}}$, it can be easily verified that $JT\subseteq TJ$.
\end{proof}

\section{Main Results}
In this section we prove the main result. The following results are useful in proving the theorem.

\begin{thm}\label{separability}
Let $H$ be a quaternionic Hilbert space and $H^{J_m}_{+}$ be the slice Hilbert space of $H$. Then
\begin{enumerate}
\item \label{basesrelation}   if ${\{e_r:r\in \mathbb N}\}$ is an orthonormal basis for $H^{J_m}_{+}$, then  ${\{e_r:r\in \mathbb N}\}$ is an orthonormal basis for $H$. On the otherhand, if ${\{\phi_r:r\in \mathbb N}\}$ is an orthonormal basis of $H$, then ${\{e_r:=\dfrac{(\phi_r-J\phi_r)}{\sqrt{2}}:r\in \mathbb N}\}$ is an orthonormal basis for $H^{J_m}_{+}$
\item \label{separabilityrelation} $H^{J_m}_{+}$  is separable if and only if $H$ is separable.
\end{enumerate}
\end{thm}
\begin{proof}
Proof of (\ref{basesrelation}):
Let $\mathcal B_{+}:={\{e_r:r\in \mathbb N}\}$ be an orthonormal basis for $H^{J_m}_{+}$ and $x\in H$. Then $x=x_{+}+x_{-}$, where $x_{\pm}\in H^{J_m}_{\pm}$.  Since the inner product on $H^{J_m}_{\pm}$ is the restriction of the inner product on $H$ to $H^{J_m}_{\pm}$, it follows that $\langle e_r,x\rangle =\langle e_r,x_{+}\rangle +\langle e_r,x_{-}\rangle $. Thus, we have $x=\displaystyle \sum_{r=1}^{\infty} e_r\, \langle e_r,x\rangle$, showing the that $\mathcal  B_{+}$ is an orthonormal basis for $H$. Hence $H$ is separable.

Assume that $\mathcal B:={\{\phi_r:r\in \mathbb N}\}$ is an orthonormal basis for $H$ and  let $\mathcal B_{+}:={\{e_r:r\in \mathbb N}\}$, where $e_r=\dfrac{(\phi_r-J\phi_r)}{\sqrt{2}},\; r\in \mathbb N$. Then
\begin{align*}
\langle e_r,e_s\rangle &=\frac{1}{2}\; \langle (I-J)(\phi_r), (I-J)(\phi_s)\rangle \\
                                   &=\frac{1}{2}\;\langle \phi_r, (I-J)^{*}(I-J)(\phi_s)\rangle \\
                                   &=\frac{1}{2}\;\langle \phi_r, (I+J)(I-J)(\phi_s)\rangle \\
                                    &=\frac{1}{2}\;\langle \phi_r, 2(\phi_s)\rangle \\
                                    &=\langle \phi_r, \phi_s\rangle.
\end{align*}
This shows that $\mathcal B_{+}$ is an othonormal set. It is clear that $\frac{1}{\sqrt{2}}(I-J):H\rightarrow H^{J_m}_{+}$ is  $\mathbb C_m$-linear,  onto isometry and hence unitary. As $e_r=\frac{1}{\sqrt{2}}(I-J)(\phi_r),\ \text{for all}\; r\in \mathbb N$, it follows that $\mathcal B_{+}$ is an orthonormal basis for $H^{J_m}_{+}$.

Proof of (\ref{separabilityrelation}): The proof follows from (\ref{basesrelation}).
\end{proof}
\begin{thm}\label{reqdproperties}
Let $T\in \mathcal C(H)$ and $T_{+}\in \mathcal C(H^{J_m}_{+})$ be such that $\widetilde{T_{+}}=T$ as in Proposition \ref{extensionrestriction}. Then the following statements hold:
\begin{enumerate}
\item \label{compact} $T$ is compact if and only if $T_{+}$ is compact
\item \label{Hilbert-Schmidt} $T$ is Hilbert-Schmidt if and only if $T_{+}$ is Hilbert-Schmidt. In this case, $\|T\|_2=\|T_{+}\|_2$.
\item \label{diagonal}  $T_{+}$ is diagonal if and only if $T$ is diagonal.
\end{enumerate}
\end{thm}
\begin{proof}
Clearly, if $T$ is compact, then $T_{+}$ is compact. The other implication follows by  the proof of \cite[Theorem 1.4]{Ghiloni1}. The proof of (\ref{Hilbert-Schmidt}) can be obtained with the following observation: If ${\{e_r:r\in \mathbb N}\}$ is an orthonormal basis for $H^{J_m}_{+}$, then it is also an orthonormal basis for $H$. Since the Hilbert-Schmidt norm does not depend on the orthonormal basis and $T{e_r}=T_{+}e_r$ for each $r\in \mathbb N$, we have,
\begin{equation*}
\|T\|_2^2=\displaystyle \sum_{r=1}^{\infty} \|Te_r\|^2=\sum_{r=1}^{\infty} \|T_{+}e_r\|^2=\|T_{+}\|_2^2.
\end{equation*}

To prove (\ref{diagonal}), let $\mathcal B_{+}={\{e_r:r\in \mathbb N}\}$ be an orthonormal basis for $H^{J_m}_{+}$ such that $T_{+}e_r=e_r\, \lambda_r$ for each $r\in \mathbb N$, where $\lambda_r\in \mathbb C_m$. Since $\mathcal B_{+}$ is also an orthonormal basis for $H$, we have that $Te_r=T_{+}e_r=e_r\lambda_r$ for each $r\in \mathbb N$. Thus $T$ is diagonalizable.

On the other hand, if $T$ is diagonal, there exists an orthonormal basis $\mathcal B:={\{\phi_r:r\in \mathbb N}\}$ and a sequence  $(\mu_r)$ of quaternions such that $T\phi_r=\phi_r\, \mu_r$ for each $r\in \mathbb N$. Let $e_r=\dfrac{\phi_r-J\phi_r}{2}$ for each $r\in \mathbb N$.
Then by Theorem \ref{separability}, $\mathcal B_{+}:={\{e_r:r\in \mathbb N}\}$ is an orthonormal basis for $H^{J_m}_{+}$. Hence
\begin{align*}
T_{+}e_r=Te_r&=\frac{T\phi_r-TJ\phi_r}{2}\\
             &=\frac{\phi_r \mu_r-T\phi_r m}{2}\\
             &=\frac{\phi_r \mu_r-\phi_r \mu_r m}{2}\\
              &=\frac{\phi_r (\mu_r-\mu_r\,  m)}{2}\\
              &=\frac{\phi_r \mu_r(1-  m)}{2}.
\end{align*}
Hence $T_{+}$ is diagonal with respect to $\mathcal B_{+}$.
\end{proof}

Here we recall the Berg's generalization of Weyl-von Neumann theorem:

\begin{thm}\cite[Corollary 2]{berg}\label{bergthmcomplex}

Let $T$ be a (not necessarily bounded) normal operator on the separable (complex) Hilbert space $H$. Then for $\epsilon>0$ there exists a diagonal operator $D$ and a compact operator $K$ with norm less than $\epsilon$ such that $T=D+K$.
\end{thm}

\begin{thm}
Let $T\in \mathcal C(H)$ be normal. Then  for every $\epsilon >0$, there exists a compact operator $K$ with $\|K\|<\epsilon$ and a diagonal operator $D$ on $H$ such that $T=D+K$.
\end{thm}
\begin{proof}
  Let $\epsilon>0$ be given. Since  $T$ is normal, there exists an anti-self-adjoint and unitary operator $J\in \mathcal B(H)$
  such that $JT\subseteq TJ$. Let $T_{+}$ be the unique operator on $H^{J_m}_{+}$ such that $\widetilde {(T_{+})}=T$. Now, by Theorem \ref{bergthmcomplex}, there exists a compact operator $K_{+}$ on $H^{J_m}_{+}$ with $\|K_{+}\|<\epsilon$ and a diagonal operator $D_{+}$ on $H^{J_m}_{+}$ such that $T_{+}=K_{+}+D_{+}$. Hence $T=D+K$, where $D:=\widetilde{D_+}$ and $K:=\widetilde{K_{+}}$. Also, note that $D$ is diagonal operator and $K$ is a compact operator with $\|K\|=\|K_{+}\|<\epsilon$.
\end{proof}

\begin{thm}\label{finitelength}
Let $H:=\ell^2(\mathbb H)$   and  $N\in \mathcal B(H)$ be normal. Let $K$ be a compact subset of $\mathbb C^{+}_m$ such that $K$ is contained in a curve of finite length. Assume that $\sigma_S(N)=\Omega_{K}$. Then there exists a diagonal operator $D$ on $\ell^2(\mathbb H)$ and for any given $\epsilon>0$, there exists a Hilbert-Schmidt operator  $K$ with the Hilbert-Schmidt norm less that $\epsilon$ such that $N=D+K$.
\end{thm}
\begin{proof}
First, observe that, if $N_{+}\in \mathcal B(H^{J_m}_{+})$ is the unique operator such that $N=\widetilde{N_{+}}$, then the spectrum $\sigma(N_{+})$ is $\sigma(N_{+})=\Omega_{K}\cap  \mathbb C^{+}_m=K$ (see \cite[Corollary 5.13]{Ghiloni} for details). Hence by \cite[Theorem 3]{berg}, $N_{+}=D_{+}+K_{+}$, where $D_{+}$ is a diagonal operator,  $K_{+}$ is a Hilbert-Schmidt operator on $H^{J_m}_{+}$. Moreover, for any $\epsilon>0$ we may select $D_{+}$ and $K_{+}$ so that the Hilbert-Schmidt norm of $K$ is less than $\epsilon$.

Let $D$ and $K$ be the unique extensions of $D_{+}$ and $K_{+}$, respectively, as in Proposition \ref{extensionrestriction}. Then $D$ is diagonal, $K$ is Hilbert-Schmidt by Theorem \ref{reqdproperties} and $\|K_{+}\|=\|K\|<\epsilon$ by Proposition \ref{extensionrestriction}.
\end{proof}

In a similar way, we can prove the following result:

\begin{thm}
Let $H=\ell^2(\mathbb H)$ and $N\in \mathcal C(H)$ be normal. Assume that $\sigma_S(N)=\Omega_{K}$, where $K$ is a subset of a rectifiable curve in $\mathbb C_m^{+}$. Then for $\epsilon>0$ there exists a diagonal operator $D$ and a Hilbert-Schmidt operator  $K$ with $\|K\|<\epsilon$ such that $N=D+K$.
\end{thm}
\begin{proof}

The proof follows by \cite[Corollary 4]{berg} and the technique used in Theorem \ref{finitelength}.
\end{proof}
\bibliographystyle{plain}

\end{document}